\documentclass{amsart}

\usepackage{graphicx,amssymb,mathrsfs,amsmath,color,fancyhdr}
\usepackage[all]{xy}
\newtheorem{theorem}{Theorem}[section]
\newtheorem{lemma}[theorem]{Lemma}

\theoremstyle{definition}

\newtheorem{corollary}[theorem]{Corollary}

\theoremstyle{remark}

\numberwithin{equation}{section}

\begin{document}
 
\title[Uniform lower bound for the lcm of a polynomial sequence]
{Uniform lower bound for the least common multiple of a polynomial sequence}


\begin{abstract}
Let $n$ be a positive integer and $f(x)$ be a polynomial with nonnegative
integer coefficients. We prove that ${\rm lcm}_{\lceil n/2\rceil \le i\le n}
\{f(i)\}\ge 2^n$ except that $f(x)=x$ and $n=1, 2, 3, 4, 6$ and that $f(x)=x^s$
with $s\ge 2$ being an integer and $n=1$, where $\lceil n/2\rceil$ denotes
the smallest integer which is not less than $n/2$. This improves and extends
the lower bounds obtained by Nair in 1982, Farhi in 2007 and Oon in 2013.
\end{abstract}
\author{Shaofang Hong}
\address{Mathematical College, Sichuan University, Chengdu 610064, P.R. China
and Yangtze Center of Mathematics, Sichuan University, Chengdu 610064, P.R. China}
\email{sfhong@scu.edu.cn; s-f.hong@tom.com; hongsf02@yahoo.com}
\author{Yuanyuan Luo}
\address{Mathematical College, Sichuan University, Chengdu 610064, P.R. China}
\email{yuanyuanluoluo@163.com}
\author{Guoyou Qian}
\address{Center for Combinatorics, Nankai University, Tianjin 300071, P.R. China}
\email{qiangy1230@163.com; qiangy1230@gmail.com}
\author{Chunlin Wang}
\address{Mathematical College, Sichuan University, Chengdu 610064, P.R. China}
\email{wdychl@126.com}
\thanks{The work was supported partially by National Science Foundation of
China Grant \#11371260, by the Ph.D. Programs Foundation of Ministry of Education
of China Grant \#20100181110073 and by Postdoctoral Science Foundation of China
Grant \#2013M530109}

\maketitle

\section{Introduction}
The least common multiple of consecutive positive
integers was first studied by Chebyshev, who made an important progress
for the proof of prime number theorem. Actually, Chebyshev \cite{[Ch]} introduced the function
$\psi(x):=\sum_{p^k\le x}\log p=\log {\rm lcm}_{1\le i\le x}\{i\}$, where $x>0$ is a real number.
From Chebyshev's work, one can derive that prime number theorem is equivalent to the statement:
$\psi(n)=\log {\rm lcm}(1, ..., n)\sim n$ as $n$ tends to infinity. Since then, the least
common multiple of sequences of integers became popular. Bateman, Kalb and Stenger \cite{[BKS]}
gave an asymptotic formula of $\log{\rm lcm}_{1\le i\le n}\{b+ai\}$ as $n$ tends to infinity,
where $a\ge 1$ and $b\ge 0$ are  coprime integers. Hong, Qian and Tan \cite{[HQT]} got an
asymptotic formula of the least common multiple of a sequence of products of linear polynomials.
Qian and Hong \cite{[QH]} investigated the asymptotic behavior of the least common multiple of
any consecutive arithmetic progression terms. Further, Farhi and Kane \cite{[FK]}
and Hong and Qian \cite{[HQ]} obtained some results on the least common multiple
of consecutive arithmetic progression terms.

Effective bounds for the least common multiple of integer sequences are given by several authors.
Hanson \cite{[Ha]} proved that ${\rm lcm}_{1\le i\le n}\{i\}<3^n$ for any integer $n\ge 1$.
Nair \cite{[N]} showed that ${\rm lcm}_{1\le i\le n}\{i\}\ge 2^n$ for any integer $n\ge 7$.
Lower bounds of the least common multiple of finite arithmetic progression
are investigated by Farhi \cite{[F1]} \cite{[F]}, Hong and Feng \cite{[HF]}, Hong and
Kominers \cite{[HK]} and Wu et al. \cite{[WTH]}. For the quadratic case,
some results are also achieved. Farhi \cite{[F]} provided a nontrivial lower bound
for ${\rm lcm}_{1\le i\le n}\{i^2+1\}$. Oon \cite{[O]} improved Farhi's lower
bound by proving that ${\rm lcm}_{1\le i\le n}\{i^2+c\}\ge 2^n$ with
$c$ being a positive integer.

In this paper, we find surprisingly that $2^n$ is the uniform lower bound for
the least common multiple of polynomial sequences of nonnegative
integer coefficients. That is, we have the following result.

\begin{theorem}\label{main}
Let $n\ge 1$ be an integer and $f(x)$ be a polynomial
of nonnegative integer coefficients. Then
${\rm lcm}_{\lceil n/2\rceil\le i\le n}\{f(i)\}\ge 2^n$
except that $f(x)=x$ and $n=1, 2, 3, 4, 6$ and that $f(x)=x^s$
with $s\ge 2$ being an integer and $n=1$, where $\lceil n/2\rceil$
stands for the smallest integer which is not less than $n/2$.
\end{theorem}

In particular, we have the following interesting result.

\begin{corollary}\label{cor1}
{\it Let $n\ge 1$ be an integer and $f(x)$ be a
polynomial of nonnegative integer coefficients.
Then ${\rm lcm}_{1\le i\le n}\{f(i)\}\ge 2^n$
except that $f(x)=x$ and $n=1, 2, 3, 4, 6$ and
that $f(x)=x^s$ with $s\ge 2$ being an integer and $n=1$.}
\end{corollary}

Evidently, if we take $f(x)=x$, then Corollary 1.2 becomes
Nair's lower bound \cite{[N]}. If one picks $f(x)=x^2+c$, then Theorem
1.1 reduces to Oon's result \cite{[O]}.

The paper is organized as follows. In Section 2, we present some basic
facts which are needed in the proof of our main result. Consequently,
in Section 3, we prove three results about the least common multiple,
and then show Theorem 1.1 as the conclusion of this paper.

\section{Preliminaries}
In this section, we show three lemmas which can be proved
with a little effort and are needed in
the proof of Theorem \ref{main}. Recall that a
complex number is called an {\it algebraic integer} if it is
a root of monic polynomial of integer coefficients
(see, for example, \cite{[AW]}).

\begin{lemma}\label{lm1}
Let $s\ge 1$ be an integer and $f(x)=\sum_{i=0}^s a_ix^i\in
\mathbb{Z}[x]$ be a polynomial of degree $s$.
If $\alpha_{1}, ..., \alpha_{s}$ are $s$ roots of $f(x)$,
then $a_s\big(\displaystyle\prod_{j\in \{1,...,s\}
\setminus\{i\}}\alpha_j\big)$ is an algebraic integer
for each integer $i$ with $1\le i\le s$.
\end{lemma}
\begin{proof}
Clearly, Lemma \ref{lm1} is true if $s=1$. We let $s\ge 2$ in what follows.
Write $\beta_i:=a_s\big(\displaystyle\prod_{j\in \{1,...,s\}
\setminus\{i\}}\alpha_j\big)$ for each integer $i$ with $1\le i\le s$.
If at least two of $\alpha_{1}, ..., \alpha_{s}$ are zero, then
$\beta_i=0$ for each integer $i$ with $1\le i\le s$. So Lemma \ref{lm1}
holds in this case. If exactly one of $\alpha_{1}, ..., \alpha_{s}$ is zero,
saying $\alpha_t=0$ for some integer $t\in \{1, ..., s\}$, then
$\beta_t=(-1)^{s-1}a_1\in \mathbb{Z}$ and $\beta_i=0$ for each integer
$i$ with $i\ne t$ and $1\le i\le s$. Hence Lemma \ref{lm1} is true in this case.

Assume now that none of $\alpha_{1}, ..., \alpha_{s}$ is zero.
Fix an integer $i$ with $1\le i\le s$. Since $a_{s}\ne 0$, one has
$\beta_i=a_s\frac{(-1)^{s}a_0/a_s}{\alpha_i}=(-1)^{s}\frac{a_0}{\alpha_i}.$
Therefore, to show that $\beta_i$ is an algebraic integer, it
suffices to prove that $\frac{a_0}{\alpha_i}$ is an algebraic
integer. From $f(\alpha_i)=0$, one derives that
$$
\frac{a_0^{s-1}}{\alpha_i^s}f(\alpha_i)=\Big(\frac{a_{0}}{\alpha_i}\Big)^{s}+
a_1\Big(\frac{a_{0}}{\alpha_i}\Big)^{s-1}+...+a_{s-1}a_{0}^{s-2}\Big(
\frac{a_{0}}{\alpha_i}\Big)+a_{s}a_{0}^{s-1}=0.
$$
This means that $\frac{a_0}{\alpha_i}$ is a root of the integer polynomial
$g(x)=x^{s}+a_{1}x^{s-1}+...+a_{s-1}a_{0}^{s-2}x +a_{s}a_{0}^{s-1}$,
from which it follows that $\frac{a_0}{\alpha_i}$ is an algebraic
integer. Lemma {\ref{lm1}} is proved in this case.
The proof of Lemma {\ref{lm1}} is complete.
\end{proof}

\begin{lemma}\label{lm2}
For any positive integer $n\ge 7$, we have
$\big\lceil n/2\big\rceil {n\choose \lceil n/2\rceil}>2^n.$
\end{lemma}
\begin{proof}
We prove Lemma \ref{lm2} by induction on $n$. Evidently,
$\lceil n/2\rceil {n\choose \lceil n/2\rceil}>2^n$ holds for $n=7$ and 8.
Now let $n\geq 7$ and we assume that $\big\lceil n/2\big\rceil {n\choose \lceil n/2\rceil}>2^n$
is true for the $n$ case. Now we consider the $n+1$ case. One can easily check that
\begin{align*}
\big\lceil (n+1)/2\big\rceil {n+1\choose \lceil (n+1)/2\rceil}
= {\left\{
\begin{array}{rl}
2\lceil n/2\rceil{n\choose \lceil n/2\rceil},\quad &\text{if}\ n\  \text{is odd,} \\
(2\lceil n/2\rceil+1){n\choose \lceil n/2\rceil},\quad &\text{if}\ n\  \text{is even.}
\end{array}
\right.}
\end{align*}
It then follows that $\big\lceil (n+1)/2\big\rceil {n+1\choose \lceil (n+1)/2\rceil}>2^{n+1}.$
Hence Lemma \ref{lm2} holds for the $n+1$ case. Lemma \ref{lm2} is proved.
\end{proof}

\begin{lemma}\label{lm3}
Let $x$ be an indeterminate and let $m$ and $n$ be positive integers such
that $m\le n$. Then we have
\begin{equation}\label{eq:2.1}
\sum\limits_{k=m}^{n}(-1)^{n-k}{n-m\choose k-m}\prod_{j=m\atop j\ne k}^n(x-j)=(n-m)!.
\end{equation}
\end{lemma}

\begin{proof}
We show (\ref{eq:2.1}) by induction on $n$.
Obviously, (\ref{eq:2.1}) is true if $n=m$. Suppose that (\ref{eq:2.1}) holds
for the $n-1$ case. Now we let $n>m$. We prove that (\ref{eq:2.1})
also holds for the $n$ case. Since $(1-1)^{n-m}=0$, we have
$\sum_{k=m}^{n-1}(-1)^{k-m}{n-m\choose k-m}=(-1)^{n-1-m}$.
So by induction hypothesis, we get that
\begin{align*}
(n-m)!&=(n-m)\sum\limits_{k=m}^{n-1}(-1)^{n-1-k}{n-1-m\choose k-m}\prod_{j=m\atop j\ne k}^{n-1}(x-j)\\
\nonumber&=(n-m)\sum\limits_{k=m}^{n-1}\frac{(-1)^{n-k}}{n-k}{n-1-m\choose k-m}
((x-n)-(x-k))\prod_{j=m\atop j\ne k}^{n-1}(x-j)\\
\nonumber&=\sum\limits_{k=m}^{n-1}(-1)^{n-k}{n-m\choose k-m}
\prod_{j=m\atop j\ne k}^{n}(x-j)-\Big(\prod_{j=m}^{n-1}(x-j)\Big)
\sum\limits_{k=m}^{n-1}(-1)^{n-k}{n-m\choose k-m}\\
\nonumber&=\sum\limits_{k=m}^{n-1}(-1)^{n-k}{n-m\choose k-m}
\prod_{j=m\atop j\ne k}^{n}(x-j)+{n-m\choose n-m}\prod_{j=m}^{n-1}(x-j)\\
\nonumber&=\sum\limits_{k=m}^{n}(-1)^{n-k}{n-m\choose k-m}
\prod_{j=m\atop j\ne k}^{n}(x-j).
\end{align*}
Therefore (\ref{eq:2.1}) is true for the $n$ case.
This finishes the proof of Lemma \ref{lm3}.
\end{proof}

\noindent{\bf Remark.} Lemma \ref{lm3} can be proved using other methods.
For example, one can prove it by using the fundamental theorem of algebra.

\section{Proof of Theorem 1.1}

In this section, we show Theorem 1.1. We begin with the following lemma.

\begin{lemma}\label{key}
Let $s\ge 1$ be an integer and $f(x)\in \mathbb{Z}[x]$ be a polynomial of
degree $s$ and with $a_s$ as its leading coefficient.
Then for any two positive integers $m$ and $n$ with $1\le m\le n$, we have
$${\rm lcm}(f(m), f(m+1), \ldots, f(n))\geq \frac{1}{(n-m)!}
\prod\limits_{k=m}^{n}\bigg|\frac{f(k)}{a_{s}}\bigg|^{\frac{1}{s}}.$$
\end{lemma}

\begin{proof}
If $f(k)=0$ for some integer $k$ with $m\le k\le n$, then Lemma \ref{key}
is clearly true. In what follows we assume that $f(k)\neq 0$ for all integers
$k$ with $m\le k\le n$.

Write $f(x)=\sum_{i=0}^s a_ix^i$. Suppose that $\alpha_{1},
..., \alpha_{s}$ are $s$ roots of $f(x)$.
Then $f(x)=a_{s}(x-\alpha_{1})...(x-\alpha_{s})$. It infers that
$h_k(x):=(-1)^sf(k-x)=a_s\prod_{i=1}^s(x-(k-\alpha_i))\in \mathbb{Z}[x]$
is also a polynomial with the leading coefficient $a_s$ and
$k-\alpha_1,...,k-\alpha_s$ are $s$ roots of $h_k(x)$ for
each integer $k$ with $m\le k\le n$. So by Lemma \ref{lm1}, we know that
$\frac{f(k)}{k-\alpha_{i}}=a_s\prod_{j\in \{1,...,s\}\setminus\{i\}}(k-\alpha_j)$
is an algebraic integer for each pair $(k, i)$ with $m\le k\le n$ and $1\le i\le s$.
It follows that ${\rm lcm}(f(m), f(m+1),\cdots, f(n))/(k-\alpha_i)$ is an
algebraic integer for each integer $i$ with $1\le i\le s$.
Since $f(k)\neq 0$ for all $m\le k\le n$, we have $k-\alpha_i\ne 0$ for all
pairs $(k, i)$ with $1\le i\le s$ and $m\le k\le n$. Then letting
$x=\alpha_i (1\le i\le s)$ in (\ref{eq:2.1}), one deduces that
\begin{equation}\label{eq:3.3}
\frac{(n-m)!}{\prod\limits_{k=m}^{n}(k-\alpha_i)}
=\sum\limits_{k=m}^{n}(-1)^{k-m}{n-m\choose k-m}
\frac{1}{k-\alpha_i}.
\end{equation}
Multiplying both sides of (\ref{eq:3.3}) by ${\rm lcm}(f(m),..., f(n))$, we obtain that
$$\mathcal{A}_i:=(n-m)!{\rm lcm}(f(m),..., f(n))\prod\limits_{k=m}^{n}\frac{1}{k-\alpha_i}$$
is a nonzero algebraic integer, and so is the product $\mathcal{A}:=\prod\limits_{i=1}^{s}\mathcal{A}_i$.
But one can easily derive that
\begin{equation}\label{eq:3.4}
\mathcal{A}=((n-m)!)^{s}({\rm lcm}(f(m),..., f(n)))^{s}\prod\limits_{k=m}^{n}\frac{a_s}{f(k)},
\end{equation}
which implies that $\mathcal{A}$ is a nonzero rational number. Thus $\mathcal{A}$ is a
nonzero rational integer and so $|\mathcal{A}|\geq 1$. This together with (\ref{eq:3.4})
concludes the desired result. The proof of Lemma \ref{key} is complete.
\end{proof}

\begin{lemma}\label{key1}
Let $f(x)$ be a polynomial of degree 2 and of nonnegative integer coefficients.
Then for any integer $m\ge 2$, we have
${\rm lcm}(f(m-1), f(m))\ge \frac{(m(m-1))^{2}}{2m-1}.$
\end{lemma}

\begin{proof}
Since for any integer $m\ge 2$, we have ${\rm gcd}(f(m-1), f(m))|(f(m)-f(m-1))$. This infers that
${\rm gcd}(f(m-1), f(m))\le f(m)-f(m-1)$. Write $f(x)=a_{2}x^{2}+a_{1}x+a_{0}\in \mathbb{Z}$$[x]$,
where $a_0, a_1\ge 0$ and $a_2\ge 1$. Then we have
\begin{align*}
{\rm lcm}(f(m-1), f(m))=&\frac{f(m)f(m-1)}{{\rm gcd}(f(m-1), f(m))}\ge \frac{f(m)f(m-1)}{f(m)-f(m-1)}\\
&=\frac{(a_{2}m^{2}+a_{1}m+a_{0})(a_{2}(m-1)^{2}+a_{1}(m-1)+a_{0})}{(2m-1)a_{2}+a_{1}}\\
&\ge \frac{m^{2}(m-1)(a_{2}(m-1)+a_{1})}{(2m-1)a_{2}+a_{1}}\\
&\ge \frac{m^{2}(m-1)(a_{2}(m-1)+a_{1})}{(2m-1)a_{2}+\frac{2m-1}{m-1}a_{1}}=\frac{(m(m-1))^{2}}{2m-1}
\end{align*}
as desired. This concludes the proof of Lemma \ref{key1}.
\end{proof}

\begin{lemma}\label{key2}
Let $a$ and $b$ be coprime positive integers. Then
${\rm lcm}(a, a+b, a+2b)=a(a+b)(a+2b)$ or $\frac{1}{2}a(a+b)(a+2b).$
\end{lemma}

\begin{proof}
Since $a$ and $b$ are coprime, we have ${\rm gcd}(a, a+b)|\gcd(a, (a+b)-a)=\gcd(a, b)$
and hence $\gcd(a, a+b)=1$. Similarly, one has $\gcd(a+b, a+2b)=1$ and
${\rm gcd}(a, a+2b)|\gcd(a, 2b)$. So ${\rm gcd}(a, a+2b)=1$ or 2. Then the desired
result follows immediately from the following well-known identity:
$${\rm lcm}(a, a+b, a+2b)
=\frac{a(a+b)(a+2b)\gcd(a, a+b, a+2b)}{\gcd(a, a+b)\gcd(a+b, a+2b)\gcd(a, a+2b)}.$$
So Lemma \ref{key2} is proved.
\end{proof}

We are now in a position to show Theorem \ref{main}.

\noindent{\it Proof of Theorem \ref{main}.} Since $f(x)$ is a polynomial with
nonnegative integer coefficients, we may let
$f(x)=a_{s}x^{s}+a_{s-1}x^{s-1}+...+a_{1}x+a_{0}\in \mathbb{Z}$$[x]$,
where $a_i\ge 0$ and $a_s\ge 1$. Then for any integer $n\ge 7$, by Lemmas
\ref{key} and \ref{lm2} and noting that $f(k)\ge a_{s}k^{s}$, we have
\begin{align*}
{\rm lcm}_{\lceil n/2\rceil \le i\le n} \{f(i)\}
& \ge \frac{\prod\limits_{k=\lceil
n/2\rceil}^{n}\big|\frac{f(k)}{a_{s}}\big|^{\frac{1}{s}}}{(n-\lceil n/2\rceil)!}
\ge \frac{\prod\limits_{k=\lceil n/2\rceil}^{n}k}{(n-\lceil
n/2\rceil)!}=\lceil n/2\rceil {n\choose \lceil n/2\rceil}> 2^n.
\end{align*}
So it remains to check that Theorem \ref{main} is true for all
positive integers $n\le 6$ in the following. First we consider the case $n=1$.
If $f(x)$ has at least two terms or $a_s\geq 2$, then
${\rm lcm}(f(\lceil n/2\rceil),..., f(n))=f(1)=\sum_{i=0}^s a_i\ge 2.$
Now let $2\le n\le 6$. We divide the proof into the following three cases.

{\sc Case 1.} $s\ge 3$. Then
${\rm lcm}(f(\lceil\frac{n}{2}\rceil),..., f(n))\ge a_{s}n^{s}\ge n^{s}\ge n^3>2^n$
for each integer $n$ with $2\le n\le 6$.
So Theorem \ref{main} is true in this case.

{\sc Case 2.} $s=2$. Then
${\rm lcm}(f(\lceil\frac{n}{2}\rceil),..., f(n))\ge f(n)\ge a_2n^2\ge n^2\ge 2^n$
for each integer $n$ with $2\le n\le4$. On the other hand, by Lemma \ref{key1}, we have
${\rm lcm}(f(\lceil\frac{n}{2}\rceil),..., f(n))
\ge {\rm lcm}(f(n-1), f(n))\ge \frac{(n(n-1))^{2}}{2n-1}\ge 2^n$
for $n=5, 6$. Theorem 1.1 is proved in this case.

{\sc Case 3.} $s=1$. First let $a_{0}=0,\ a_1=1 $ and $n=5$. Then $f(x)=x$.
Hence ${\rm lcm}_{\lceil \frac{n}{2}\rceil\le i\le n}\{f(i)\}
={\rm lcm}_{\lceil \frac{5}{2}\rceil\le i\le 5}\{i\}
={\rm lcm}(3, 4, 5)=60>2^5$ as required.
Now let $a_0=0 $ and $a_1 \ge 2$. Then $f(x)=a_1 x.$
Denote $\mathcal{L}_n:={\rm lcm}_{\lceil \frac{n}{2}\rceil\le i \le n}\{i\}$.
It is well known that $\mathcal{L}_n\ge 2^{n-1}$.
Since $a_1 \ge 2$ and $a_0=0$, one has
${\rm lcm}_{\lceil \frac{n}{2}\rceil\le i \le n}\{f(i)\}
=a_1 \mathcal{L}_n\ge 2\mathcal{L}_n\ge 2^n$ as claimed.

Finally, let $a_0 \ge 1, a_1\ge 1$ and $\gcd(a_0, a_1)=d$.
One may write $a_0=da$ and $a_1=db$ for some coprime
positive integers $a$ and $b$. If $n=2$ and 3, then we have
\begin{align*}
&{\rm lcm}_{\lceil \frac{n}{2}\rceil\le i \le n}\{f(i)\}
={\rm lcm}(f(n-1), f(n))=\frac{f(n-1)f(n)}{\gcd(f(n-1), f(n))}\\
&=\frac{f(n-1)f(n)}{\gcd(f(n)-f(n-1), f(n))}
=\frac{(a_1(n-1)+a_0)(a_1 n+a_0)}{\gcd(a_1, a_1 n+a_0)}\\
&=d(a(n-1)+b)(an+b)\ge n(n+1)>2^n
\end{align*}
as required.
If $4\le n\le 6$, then by Lemma \ref{key2}, one has
\begin{align*}
&{\rm lcm}_{\lceil \frac{n}{2}\rceil\le i \le n}\{f(i)\}
\ge {\rm lcm}(f(n-2), f(n-1), f(n))\\
&=d\cdot{\rm lcm}(a(n-2)+b, a(n-1)+b, an+b)\\
&\ge \frac{1}{2}d(a(n-2)+b)(a(n-1)+b)(an+b)\ge \frac{1}{2}(n-1)n(n+1)>2^n
\end{align*}
as desired.

This completes the proof of Theorem \ref{main}. \hfill$\Box$

\begin{center}
{\sc Acknowledgement}
\end{center}
The authors would like to thank the anonymous referee for
very helpful comments and suggestions that improved its presentation.

\bibliographystyle{amsplain}

\end{document}